\documentclass[a4paper,12pt,reqno]{amsart}
\usepackage{amsfonts}
\usepackage{amsmath}
\usepackage{mathrsfs}
\usepackage{amssymb}
\usepackage{amsthm}
\usepackage{amscd}

\usepackage{color}
\definecolor{bluScuro}{RGB}{0, 102, 204}

\usepackage[all]{xy}
\usepackage{tikz}

\usepackage[
draft=false,
colorlinks, citecolor=darkgreen,
pdfauthor={L.Fassina,G Pirola},
pdftitle={},
linktocpage
]{hyperref}
\hypersetup{citecolor=blue,linktocpage}

\numberwithin{equation}{section}

\newtheorem{theorem}{Theorem}[section]

\theoremstyle{definition}
\newtheorem{remark}[theorem]{Remark}
\newtheorem{definition}[theorem]{Definition}

\newcommand{\nc}{\newcommand}
\nc{\cL}{{\mathcal L}}
\nc{\cM}{{\mathcal M}}
\nc{\bC}{{\mathbb C}}

\nc{\bN}{{\mathbb N}}
\nc{\bP}{{\mathbb P}}

\nc{\OO}{\mathcal{O}}

\usepackage{xcolor} %for colours

\begin{document}

\title[Normal functions and the theta divisor]{Normal functions, even theta
characteristics and the theta divisor}

\date{\today}

\author[I. Biswas]{Indranil Biswas}

\address{Department of Mathematics, Shiv Nadar University, NH91, Tehsil
Dadri, Greater Noida, Uttar Pradesh 201314, India}

\email{indranil.biswas@snu.edu.in, indranil29@gmail.com}

\author[L. Fassina]{Lorenzo Fassina}

\address{Dipartimento di Matematica, Universit\`a degli Studi di Pavia, Via Ferrata 5, 27100 Pavia, Italy}

\email{lorenzo.fassina02@universitadipavia.it}

\author[G. P. Pirola]{Gian Pietro Pirola}

\address{Dipartimento di Matematica, Universit\`a degli Studi di Pavia, Via Ferrata 5, 27100 Pavia, Italy}

\email{gianpietro.pirola@unipv.it}

\subjclass[2020]{14H40, 47A10, 53C27, 14H42, 14H15}

\keywords{Theta characteristic, spin Laplacian, normal function, theta divisor}

\begin{abstract}
Let $[C]$ be a point in the moduli space of curves $\mathcal{M}_g$, with $g>1$, lying outside a countable union of closed nowhere--dense subsets of an analytic parameter space
for $\mathcal M_g$.
% Let $[C]$ be a general point in the moduli space of curves $\mathcal{M}_g$, with $g\, >\,
% 1$.
Let $G \,\subset\, J(C)$ be a connected compact subgroup of real dimension $1$ of
the Jacobian of $C$, and let $L$ be an even theta characteristic on $C$. We prove that $\{\zeta
\, \in\, G\ \big\vert\,\ H^0(C,\, L \otimes \zeta)\,\not=\, 0 \}\, =\, \emptyset$ if
and only if $L \otimes \zeta_G$ is an even theta characteristic on $C$, where $\zeta_G$
is the unique non-trivial point of $G$ of order two.
\end{abstract}

\maketitle

\section*{Introduction}

Theta characteristics on compact Riemann surfaces are classical objects of study in algebraic geometry
and they play a crucial role in the theory of spin structures. Recall that for a compact connected Riemann
surface $C$ of genus $g \,\ge\, 2$, a \emph{theta characteristic} on $C$ is a
holomorphic line bundle $L$ on $C$ such that 
\[
L^{\otimes 2}\ \simeq\ \omega_C,
\]
where $\omega_C$ denotes the canonical line bundle of $C$ (see Definition \ref{def:theta_char}). These objects
are naturally identified with spin structures on the underlying Riemann surface. More specifically,
if $\Sigma$ is a
compact Riemann surface equipped with a spin structure, then the associated complexified spinor bundle
$S_{\mathbb{C}}$ on $\Sigma$ decomposes
as $S_+ \oplus S_-$ where $S_+$ is a holomorphic line bundle satisfying the condition
that $(S_+)^{\otimes 2}$ is holomorphically isomorphic to the holomorphic cotangent bundle
$\omega_\Sigma$ of $\Sigma$. Thus, $S_+$ is a theta characteristic of
$\Sigma$. Conversely, every theta characteristic of $\Sigma$ uniquely determines a spin structure
on $\Sigma$ \cite{Atiyah}.

A recent work by Adve and Giri \cite{A-A-G-Vikram} investigates the spectral properties of the Dirac operator 
$\mathcal{D}$ acting on the sections of the spinor bundle over a compact Riemann surface. Let $\Sigma$ be a
compact Riemann 
surface with a fixed spin structure, and let $S_+$ be the corresponding theta characteristic on $\Sigma$. Let 
$\lambda_0^{\mathrm{spin}}(\Sigma)$ denote the smallest eigenvalue of the spin Laplacian 
$-\Delta^{\mathrm{spin}} = -\mathcal{D}^2$. The \emph{spectral gap} for the Dirac operator is defined as 
$\sqrt{\lambda_0^{\mathrm{spin}}(\Sigma)}$. Let
\[
\widehat{H_1(\Sigma,\, \mathbb{Z})}\ =\ \mathrm{Hom}(H_1(\Sigma,\, \mathbb{Z}),\ {\rm U}(1))
\]
be the character torus (the Pontryagin dual of the first homology group); it parametrizes the
flat unitary line bundles on $\Sigma$. For each character $\chi \,\in\, \widehat{H_1(\Sigma,\, \mathbb{Z})}$,
the corresponding flat holomorphic line bundle on $\Sigma$ is denoted by $M_\chi$.

In this set-up, the spectral theory of the spin Laplacian is linked to the geometry of the twisted bundles 
$S_+ \otimes M_\chi$. More specifically, the vanishing of the first eigenvalue of the spin Laplacian is equivalent 
to the existence of non-zero holomorphic sections of these twisted bundles; indeed, a section of $S_+ \otimes 
M_\chi$ is holomorphic if and only if it lies in the kernel of the Dirac operator (cf. \cite[Lemma 
2.6]{A-A-G-Vikram}).

A key role in \cite{A-A-G-Vikram} is played by subtori of the character torus. If $T \,\subset\,
\widehat{H_1(\Sigma,\, \mathbb{Z})}$ is a closed subtorus of real dimension 1 with $T[n]\, \subset\, T$
being its subgroup of $n$-torsion points, one considers the sequence of finite abelian covers
$\Sigma_{T[n]}$ associated with the dual groups $T[n]^*$. It is shown in \cite[Proposition 2.3]{A-A-G-Vikram}
that the existence of a uniform spectral gap along such a sequence of coverings is equivalent to the
following vanishing condition:
\begin{equation}\label{ct}
H^0(\Sigma,\, S_+ \otimes M_\chi)\ =\ 0 \quad \text{for all }\ \chi \ \in\ T.
\end{equation}
Moreover in \cite{A-A-G-Vikram} the authors provide an explicit construction of a genus $2$ curve and a specific 1-dimensional subtorus of characters ensuring a uniform spectral gap for the Dirac operator along the corresponding tower of covers.
These analytic results naturally lead to an algebro-geometric question, which is explained below.

Let
$C$ be a smooth complex projective curve of genus $g \,\ge\, 2$, and set $L\,:=\, S_+$. The Jacobian $J(C) \,=
\, \mathrm{Pic}^0(C)$ parametrizes the holomorphic line bundles of degree zero on $C$. For a fixed theta
characteristic $L$ on $C$, the subscheme
\[
\Theta_L\ =\ \{ \xi\, \in\, J(C)\ \, \big\vert\ \ H^0(C,\, L \otimes \xi)\, \not=\, 0 \}\ \subset\
J(C)
\]
is the symmetric theta divisor associated with $L$. Thus, the condition on $T$ in \eqref{ct} is
equivalent
to the condition that the subtorus $T$ avoids the theta divisor $\Theta_L$.

Here we investigate this problem for real subtori of the Jacobian. Let $G\, \subset\, J(C)$ be
a connected compact subgroup of real dimension one (a real 1-torus). For a given theta characteristic $L$, we define the intersection locus:
\[
G_L\ :=\ \{ \zeta\,\, \in\,\, G\ \, \big\vert\,\ H^0(C,\, L \otimes \zeta)\, \not=\, 0 \}\ =\ G \cap \Theta_L.
\]
This $G_L$ parametrizes the twists of $L$ along $G$ that admit non-zero holomorphic sections.

Our main
objective is to identify the conditions that ensure emptiness of this set.
We show that
this property is determined by the parity of a specific theta characteristic related to the 2-torsion
subgroup of $G$.

\begin{theorem}[{Theorem \ref{Main th}}]\label{ti}
Let $[C]$ be a point in the moduli space of curves $\mathcal{M}_g$ lying outside a countable union of closed nowhere--dense subsets of an analytic parameter space
for $\mathcal M_g$, with $g\, >\, 1$. Let $G \,\subset\, J(C)$ be a connected compact subgroup of real dimension $1$, and let $L$ be an even theta characteristic on $C$. Then 
\[
G_L\ \, =\ \, \emptyset
\]
if and only if $L \otimes \zeta_G$ is an even theta characteristic on $C$, where $\zeta_G$
is the unique non-trivial point of $G$ of order two (see Definition \ref{Def_L-even}).
\end{theorem}

Theorem \ref{ti} provides a geometric criterion for the vanishing condition arising in the study
of Dirac operators. Specifically, it exactly describes the real one-dimensional subtori of the Jacobian
that fail to intersect the theta divisor. In this sense, Theorem \ref{ti} serves as an algebro-geometric
counterpart to the spectral gap conditions established in \cite{A-A-G-Vikram}.

\vspace{1em}
 
\subsection*{Structure of the paper}

The paper is organized as follows. In Section \ref{sec:1}, we recall the basic notation for projective 
curves, introduce the Jacobian variety and the theta divisor. In Section \ref{sec:2}, we provide the 
necessary background on the relative Picard bundle and spin sections, and also recall the definition and the rank 
of normal functions. We also highlight a key result concerning the local constancy of normal functions which 
serves as a fundamental tool for our analysis (see Theorem \ref{coro-spin}). Section \ref{sec:3} is devoted to the 
proof of Theorem \ref{Main th}, providing the $L$--even criterion for the vanishing of the intersection locus 
$G_L\, =\, G \cap \Theta_L$ under the genericity condition stated above.

\section{Preliminaries}\label{sec:1}

Throughout this paper, the base field is assumed to be the field of complex numbers $\mathbb{C}$. Let $C$
be an irreducible smooth projective curve. Denote by $\omega_C$ the canonical line bundle of $C$, so
\[
g\ :=\ \dim_{\mathbb{C}} H^0(C, \,\omega_C)
\]
coincides with the genus of $C$. Any divisor associated with $\omega_C$ is called a
\emph{canonical divisor}; such a divisor is denoted by $K_C$, and
the equality $\deg(K_C) \,=\, 2g - 2$ holds.

Let $\mathrm{Div}(C)$ denote the group of divisors on $C$ and $\mathrm{Div}^d(C)\, \subset\,
\mathrm{Div}(C)$ the subset of divisors of degree $d$. For any integer $d$, the \emph{Picard variety} of
degree $d$ is defined to be
\[
\mathrm{Pic}^d(C)\ :=\ \mathrm{Div}^d(C) / \sim,
\]
where $\sim$ denotes the linear equivalence. Equivalently, $\mathrm{Pic}^d(C)$ parametrizes
all the isomorphism classes of holomorphic line bundles of degree $d$ on $C$.

Let $H^0(C,\, \omega_C)^*\,=\, H^1(C,\, {\mathcal O}_C)$ denote the dual vector space of holomorphic
differentials on $C$. Integration of holomorphic $1$-forms along closed paths in $C$ defines a full lattice
\[
\Lambda\ \subset\ H^0(C,\, \omega_C)^*
\]
given by the image of the integral homology group $H_1(C,\, \mathbb{Z})$. The quotient
\[
J(C)\ :=\ H^0(C,\, \omega_C)^* / \Lambda
\]
is a complex torus of dimension $g$, which is called the \emph{Jacobian variety} of the curve
$C$. By the Abel--Jacobi theorem, there is a canonical isomorphism $J(C)\, \simeq\, \mathrm{Pic}^0(C)$.

For a line bundle $L$ on $C$, we denote the dimension of its space of global holomorphic sections by
\[
h^0(C,\, L)\ :=\ \dim_{\mathbb{C}} H^0(C,\, L).
\]

For every integer $d \,\ge\, 1$, let $C^{(d)}\ :=\ C^d / \mathfrak{S}_d$ be the symmetric product
of $C$, where $\mathfrak{S}_d$ is the symmetric group, and consider the Abel map
\[
 A_d\ \colon\ C^{(d)}\ \longrightarrow\ \mathrm{Pic}^d(C), \qquad \{p_1, \cdots ,\, p_d\}\ \longmapsto
\ \mathcal{O}_C(p_1 + \cdots + p_d).
\]
Consider the image $W_d(C)\,:=\, A_d(C^{(d)})$ of the Abel map $A_d$. Intrinsically, it can be defined as
\begin{equation}\label{ew}
W_d(C)\ :=\ \{ L \, \in\, \mathrm{Pic}^d(C)\, \ \big\vert\, \ h^0(C,\, L)\, >\, 0\}.
\end{equation}
In particular, for $d \,= \, g - 1$, the image $W_{g-1}(C)$ is a divisor in $\mathrm{Pic}^{g-1}(C)$. 

After fixing a base point $p_0 \,\in\, C$, the Abel--Jacobi theorem provides an identification
$\mathrm{Pic}^{g-1}(C) \, \simeq\, J(C)$ that sends any $L\, \in\, \mathrm{Pic}^{g-1}(C)$
to $L\otimes{\mathcal O}_C(-(g-1)p_0)\, \in\, J(C)$. Under this identification, the divisor
$W_{g-1}(C)$ translates to a divisor in the Jacobian, called the \emph{theta divisor} (a
classical treatment of the theta divisor can be found in \cite[Chap.~2, Sec.~7]{GHarris});
this theta divisor is denoted by $\Theta$. More
canonically, there exists a divisor class $\ell \,\in\, \mathrm{Pic}^{g-1}(C)$ (the 
\emph{Riemann constant}) such that the theta divisor is given by
\[
\Theta\ :=\ W_{g-1}(C) - \ell\ \subset\ J(C);
\]
it should be clarified that $\Theta$ depends on some choices, while $W_{g-1}(C)$ is completely canonical.

\begin{definition}\label{def:theta_char}
A \emph{theta characteristic} on an irreducible smooth projective curve $C$ of genus
$g$ is a line bundle $L \,\in\, \mathrm{Pic}^{g-1}(C)$ such that
\[
L^{\otimes 2}\ \,\simeq\ \,\omega_C.
\]
A theta characteristic $L$ is called \emph{even} (respectively, \emph{odd}) if $h^0(C,\, L)$ is
even (respectively, odd).
\end{definition}

\begin{remark}
The set of all theta characteristics forms a principal homogeneous space (a torsor) under the translation
action of the subgroup $J(C)[2]\, \subset\, J(C)$ given by the $2$--torsion points of the Jacobian.
\end{remark}

Theta characteristics are closely related to the geometry of the theta divisor. Given a theta
characteristic $L$, tensoring by $L^{-1}$ induces an isomorphism
\[
\mathrm{Pic}^{g-1}(C)\ \longrightarrow\ J(C), \qquad M \ \longmapsto M\ \otimes L^{-1}.
\]
Under this map, $W_{g-1}(C)$ (see \eqref{ew}) translates to a divisor in $J(C)$ --- which is
denoted by $\Theta_L$ --- that has the following description:
\begin{equation}\label{theta L}
\Theta_L\ :=\ W_{g-1}(C) - L \ =\ \{\xi\, \in\, J(C)\,\ \big\vert\,\ H^0(C,\, \xi\otimes L)\, \not=\, 0\}
\ \subset\ J(C).
\end{equation}

The divisor $\Theta_L$ in \eqref{theta L} is called the \emph{symmetric theta divisor} associated with the
theta characteristic $L$. It is symmetric with respect to the canonical involution
$\zeta \, \longmapsto\, \zeta^{-1}$ of the Jacobian $J(C)$, which is in fact an immediate consequence
of the Serre duality; note that for any $\xi\,\in\, J(C)$, we have $H^0(C,\, \xi\otimes L)\, \not=\, 0$
if and only if $H^1(C,\, \xi\otimes L)\, \not=\, 0$. For more on theta characteristics we refer the reader to
\cite[Chap.~6]{Mumford-Theta} and \cite{Cornalba-theta-char}.

\section{Normal functions and spin sections}\label{sec:2}

In this section we introduce the main objects appearing in Theorem \ref{coro-spin}, namely the relative 
Picard bundle, spin sections and normal functions.
Although these concepts are classical, they are briefly recalled in a minimal form that suffices for our
purpose. A comprehensive account of the Picard bundle can be found in \cite[Chap.~4, Section~2]{ACG1}, and
more details on normal functions can be found in \cite[Chap.~2, Sec.~7]{Voisin}, \cite[Chap.~9]{Carlson};
we refer to \cite[Section~3]{Hain} for the rank of a normal function.

\subsection{Picard bundle and spin sections}

Let
\begin{equation}\label{e1}
\pi\ :\ \mathcal{C}\ \longrightarrow\ Y
\end{equation}
be a smooth algebraic family of complete complex connected smooth curves of 
genus $g$ with $Y$ being a connected complex variety. For every point $y\, \in\, Y$,
the fiber $C_y\, :=\, \pi^{-1}(y)$ is a connected smooth projective curve of genus $g$. Denoting the moduli space
of smooth curves of genus $g$ by $\mathcal{M}_g$,
\begin{equation}\label{e2}
m\ :\ Y\ \longrightarrow\ \mathcal{M}_g, \qquad m(y)\, =\, [C_y]
\end{equation}
is the natural classifying map for the family $\pi$ in \eqref{e1}. We will be mostly
considering families of curves for which the map $m$ in \eqref{e2} is actually
dominant, which means that the fibers of $\pi$ are general in moduli.

The \emph{relative Picard bundle} for the family $\pi$ in \eqref{e1} is
\[
\mathrm{Pic}(\pi)\,\ :=\,\ R^1 \pi_* \mathcal{O}_{\mathcal{C}}^*.
\]
Its fiber over any $y \,\in\, Y$ is the Picard group $\mathrm{Pic}(C_y)$ of the fiber $C_y$.
Denote by $\mathrm{Pic}^d(\pi) \, \subset\, \mathrm{Pic}(\pi)$ the sub-fiber bundle whose fiber
over any $y \,\in\, Y$ is the connected component of $\mathrm{Pic}(C_y)$ that parametrizes the
isomorphism classes of the line bundles of degree $d$ on $C_y$.

A particularly important section of $\mathrm{Pic}(C_y)$ is the \emph{canonical section} $\kappa$ given by
the relative canonical bundle for $\pi$; in other words,
\begin{equation}\label{e3}
\kappa\ :\ Y\ \longrightarrow\ \mathrm{Pic}^{2g-2}(\pi), \qquad y\ \longmapsto\ K_{C_y},
\end{equation}
where $K_{C_y}$ is the canonical line bundle of the curve $C_y$.

\begin{definition}
A section $\psi\,:\, Y \,\longrightarrow \,\mathrm{Pic}^{g-1}(\pi)$ of
is called a \emph{spin section} if
\[
2 \cdot \psi\ =\ \kappa
\]
(see \eqref{e3}), or equivalently, the line bundle $\psi(y)$ on $C_y$ is a theta characteristic
for every $y \,\in\, Y$. The
parity of $\psi$ is defined as the parity of $\psi(y)$; since $Y$ is connected,
from \cite[p.~184, Theorem]{Mu} it follows immediately that $\psi(y)$ is actually independent of $y$.
\end{definition}

\subsection{Normal functions}

Let $\mathcal H\,:=\,R^1\pi_*\mathbb C\otimes\mathcal O_Y$
be the flat bundle associated with the local system
$R^1\pi_*\mathbb C$. It is endowed with the Gauss--Manin connection
\[
    \nabla\ :\ \mathcal H\ \longrightarrow\ \mathcal H\otimes\Omega_Y^1,
\]
whose flat sections are precisely the local sections of
$R^1\pi_*\mathbb C$.
Let $\mathcal H^{1,0}\,:=\,\pi_*\omega_{\mathcal C/Y}\subset\mathcal H$
be the Hodge subbundle. The Jacobian fibration associated with the
family $\pi\,:\,\mathcal C\,\longrightarrow\, Y$ is
\[
    J(\pi)\ :=\
    \frac{\mathcal H}
    {\mathcal H^{1,0}+R^1\pi_*\mathbb Z} \ \longrightarrow\ Y.
\]
Its fiber over any $y\,\in \,Y$ is
\[
    J(C_y)\ =\
    \frac{H^1(C_y,\,\mathbb C)}
    {H^{1,0}(C_y)+H^1(C_y,\,\mathbb Z)}
    \simeq \operatorname{Pic}^0(C_y).
\]

The degree-zero component $\mathrm{Pic}^0(\pi)$ of the Picard bundle coincides with the Jacobian fibration
\begin{equation}\label{ej}
J(\pi)\,\ \simeq \,\ \mathrm{Pic}^{0}(\pi)\, \ \longrightarrow\,\ Y,
\end{equation}
Holomorphic sections of $J(\pi)$ are called \emph{normal functions}.

A basic example of a normal function that will be used later is the following. Let
\begin{equation}\label{e4}
\psi\ :\ Y\ \longrightarrow\ \mathrm{Pic}^{g-1}(\pi)
\end{equation}
be a global section of the Picard bundle of degree $(g-1)$ line bundles. Then we define the associated
\emph{normal function}
\begin{equation}\label{associated normal fun}
    \nu\ \, :=\ \, 2\psi - \kappa,
\end{equation}
where $\kappa$ is constructed in \eqref{e3}; it measures the deviation of $\psi$ from being a spin section. 

To define the rank of a normal function, recall that $J(\pi)$ is diffeomorphic to a real torus bundle
\begin{equation}\label{diffeo-tori}
f\ :\ J(\pi)\ \xrightarrow{\,\,\,\sim\,\,\,}\ \frac{R^1\pi_*\mathbb{R}}{R^1\pi_*\mathbb{Z}}.
\end{equation}
Here and in what follows, \emph{flat} refers to the natural flat structure on
$R^1\pi_*\mathbb{R}$ induced by the Gauss--Manin connection, and to the corresponding
flat structure on the quotient
$\frac{R^1\pi_*\mathbb{R}}{R^1\pi_*\mathbb{Z}}$.
Over a simply connected open subset $U \,\subset\, Y$, the above flat torus bundle
$\frac{R^1\pi_*\mathbb{R}}{R^1\pi_*\mathbb{Z}}$ is trivial; hence the restriction of
$J(\pi)$ to $U$ is isomorphic to $T^{2g}\times U$ as $C^\infty$ bundles, where
$T^{2g}$ is a real torus of dimension $2g$. Denoting by $p_1$ the natural projection
$T^{2g}\times U\longrightarrow T^{2g}$, consider the composition of maps
\[
J(\pi)_{|_U}\ \xrightarrow{\,\,\,f\vert_U\,\,\,}\ T^{2g} \times U\ \xrightarrow{\,\,\,p_1\,\,\,}\ T^{2g},
\]
where $f$ is the map in \eqref{diffeo-tori}, and denote
\begin{equation}\label{phi}
\phi\,\ :=\,\ p_1 \circ (f\big\vert_U)\,\ :\,\ J(\pi)_{|_U}\,\ \longrightarrow\,\ T^{2g}.
\end{equation}

The identification in \eqref{diffeo-tori}, together with the local trivialization of
$\frac{R^1\pi_*\mathbb{R}}{R^1\pi_*\mathbb{Z}}$ over
$U$, determines a natural foliation $\mathcal{F}$ of $J(\pi)$. Indeed, for every
$a\,\in\, T^{2g}$, the submanifold
\[
    \mathcal F_a\,\ :=\,\ \phi^{-1}(a)
\]
is called a \emph{locally constant leaf}. These local leaves are
compatible with changes of local trivialization and therefore define a
canonical foliation $\mathcal F$ of $J(\pi)$.

\begin{remark}\label{remark: complx foliation}
    The foliation $\mathcal{F}$, is defined using the underlying real local
system but its leaves carry a natural complex structure. To see this, fix
$y_0\,\in\, U$ and identify, through the chosen local trivialization,
\[
    T^{2g}\ \simeq\
    \frac{H^1(C_{y_0},\,\mathbb R)}
         {H^1(C_{y_0},\,\mathbb Z)}.
\]
Let $a\,\in\, T^{2g}$ and choose a representative
$v_0\,\in\, H^1(C_{y_0},\, \mathbb R)$. Its flat extension defines a section
$v$ of $R^1\pi_*\mathbb R|_U$, whose complexification is a flat section
of $\mathcal H|_U$. Since the holomorphic structure of $\mathcal H$ is
induced by the $(0,\,1)$-part of the Gauss--Manin connection,
\[
\overline{\partial}_{\mathcal H}v\ =\ \nabla^{0,1}v\ =\ 0.
\]
Hence $v$ is holomorphic. Since
$\mathcal H^{1,0}\,\subset\,\mathcal H$ is a holomorphic subbundle, its
class defines a holomorphic section
\[
s_a\ :\ U\ \longrightarrow\ J(\pi)|_U,
    \qquad y\ \longmapsto\ [v_y].
\]
Its image is precisely $\mathcal F_a$. Therefore every locally constant leaf carries a natural structure of
complex manifold and is locally an embedded complex submanifold of
$J(\pi)$. (See \cite[p.~893, Theorem 3.1]{Bi}.) 
\end{remark}

\begin{definition}
Let $\nu\,:\, Y\, \longrightarrow\, J(\pi)$ be a normal function. Define
the rank of $\nu$ to be
\[
\mathrm{rk} (\nu)\ =\ \frac{\dim((\phi\circ \nu )(U))}{2},
\]
where $\dim$ is the real dimension, and $\phi$ is the map in \eqref{phi}. A normal function $\nu$ is said to be
\emph{locally constant} if we have $\mathrm{rk}(\nu)\, =\, 0$.
\end{definition}

\begin{remark}
Following Hain, \cite[Section~3]{Hain}, the rank of $\nu$ is an integer because the fibers of $\phi\circ \nu$
are complex subvarieties of $J(\pi)\big\vert_U$ (they are intersections of $\nu(Y)$ with the leaves of the
foliation $\mathcal F$) implying that the dimension of the image of $\phi\circ \nu$ is even.
\end{remark}

\begin{remark}
An important class of examples of a locally constant normal function is given by the torsion sections.
\end{remark}

The following result proved in \cite{FassinaPirola} on the local constancy of normal functions will be used here.

\begin{theorem}[{\cite[Corollary 6.4]{FassinaPirola}}]\label{coro-spin}
Take $\pi$ as in \eqref{e1} with $g\, \geq\, 2$ such that the corresponding map $m$ in \eqref{e2} is dominant.
Let $\psi$ be a section of the Picard bundle $\mathrm{Pic}^{g-1}(\pi)$ such that
$\dim H^0(C_y,\, \psi(y)) \,>\, 0$ for a general point $y \,\in\, Y$.
For the associated normal function
\[
    \nu\ :=\ 2\psi - \kappa ,
\]
if $\mathrm{rk}(\nu) = 0$, in other words, if $\nu$ is locally constant, then
$\psi$ is actually an odd spin section.
\end{theorem}

\section{One-dimensional real subtori and the theta divisor}\label{sec:3}

Let $C$ be an irreducible smooth complex projective curve of genus $g$, with $g \,\geq\, 2$. Let $G$ be a connected
compact subgroup of real dimension 1 of the Jacobian $J(C) \,=\, \mathrm{Pic}^0(C)$ of $C$. Such subgroups are
real subtori of the complex abelian variety $J(C)$; see \cite[Chapter~2]{Mumford-abelian}. Take a theta
characteristic $L$ of $C$. Set
\begin{equation}\label{e5}
G_L\ :=\ \{\zeta\, \in\, G\,\,\big\vert\,\, h^0(C,\, L\otimes \zeta)\, >\, 0\}\ =\ G \cap \Theta_L,
\end{equation}
where $\Theta_L$ is
symmetric theta divisor defined in \eqref{theta L}. We are seeking conditions that guarantee
the emptiness of the locus $G_L$ in \eqref{e5}. Observe that a necessary condition
for $G_L\,=\,\emptyset$ is that $h^0(L) \,=\, 0$, as otherwise the origin
${\mathcal O}_C\, \in\, G_L \, \subset\, J(C)$ belongs to $G_L$. Also, if $G_L\,=\,\emptyset$, and
$\zeta_G$ is the (unique) non-trivial two-torsion point of $G$ (recall that $\dim_{\mathbb R}G \,=\,
1$), then we have $h^0(L\otimes \zeta_G)\,=\, 0$.

\begin{definition}\label{Def_L-even}
We say that $G$ is $L$--even (respectively, $L$--odd) if $L' \,:=\, L \otimes \zeta_G$ is an even (respectively,
odd) theta characteristic on $C$, where $\zeta_G$ as before is the unique non-trivial point of $G$ of order two.
\end{definition}

\begin{remark}Note that $L'\,=\,L\otimes \zeta_G$ is a theta characteristic of $C$. We have
$\xi\, \in\, G_{L'}$ if and only if $\xi\otimes \zeta_G\, \in\, G_L$.
\end{remark}

\begin{theorem}\label{Main th}
Let $[C]$ be a point in the moduli space of curves $\mathcal{M}_g$ lying outside a countable union of closed nowhere--dense subsets of an analytic parameter space for $\mathcal M_g$, with $g\, >\, 1$.  Let $G$ be a
connected compact subgroup of real dimension 1 of the Jacobian $J(C)$ of $C$. Take an even
theta characteristic $L$ of $C$. Then $G_L$ defined in \eqref{e5} is empty if and only
if $G$ is an $L$--even torus.
\end{theorem}
 
\begin{proof}
If $G_L$ defined in \eqref{e5} is empty, then $G$ is evidently an $L$--even torus. We will
prove the converse. So assume that the torus $G$ is $L$--even.

For a general curve $C'$, we have $H^0(C',\,L')\,=\,0$ for every even theta-characteristic $L'$ on
$C'$ (see \cite[Theorem 1.10]{harris_theta}).
Let $\pi \, :\, \mathcal{C}\, \longrightarrow\, Y$ be a smooth family of complex projective curves, of genus
$g$, such that corresponding classifying map $m \, :\, Y \,\longrightarrow\, \mathcal{M}_g$ (see \eqref{e2})
is dominant. Fix a base point $0 \,\in\, Y$, and set $C \,:=\, C_0$. Let $J(\pi)$ be the
Jacobian fibration constructed in \eqref{ej}.

Over any simply connected open subset $U \,\subset\,
Y$ containing the base point $0$, the Jacobian fibration trivializes smoothly via the
diffeomorphism \eqref{diffeo-tori} as
\begin{equation}\label{ej2}
J(\pi)\big\vert_U \ \simeq\ T^{2g}\times U,\qquad T^{2g}\ =\ J(\pi)_0\ =\ (S^1)^{2g}.
\end{equation}
Let
\begin{equation}\label{e6}
G(\pi)\, \ \subset\,\ J(\pi)
\end{equation}
be a family of real one-dimensional subtori. We assume that $G(\pi)$ is flat with respect to the natural
flat connection induced by the trivialization \eqref{diffeo-tori}. Equivalently, $G(\pi)$
is a locally constant family of tori, meaning that for a local trivialization of $J(\pi)$
as in \eqref{ej2}, the subbundle $G(\pi)$ is of the form
\[
G(\pi)\big\vert_U\, \;\simeq\, \; {\mathbb S}^1 \times U\,\; \subset\,\; T^{2g} \times U,
\]
where ${\mathbb S}^1\, \subset\, T^{2g}$ is a subgroup isomorphic to $S^1$.

Using the trivialization $J(\pi)\big\vert_U \simeq T^{2g}\times U$
in \eqref{ej2}, the real one-dimensional subtori of
$T^{2g}$ are parametrized by nonzero vectors
\[
v\in \operatorname{Hom}\bigl(\mathbb Q,H_1(T^{2g},\mathbb Q)\bigr)
\simeq \mathbb Q^{2g},
\]
up to scaling. We denote by $S_v^1\subset T^{2g}$
the real one-dimensional subtorus determined by $v$.
The corresponding flat family of subtori of the Jacobian fibration is
defined by
\[
G_v(\pi)\big\vert_U
:=
S_v^1\times U
\subset
T^{2g}\times U
\simeq
J(\pi)\big\vert_U.
\]
Thus $G_v(\pi)\to U$ is a family of real one-dimensional subtori, flat
with respect to the natural flat connection on $J(\pi)\big\vert_U$.
For every $y\in U$, we denote its fiber by
\[
G_{v,y}
:=
\bigl(G_v(\pi)\bigr)_y
\subset
J(C_y)=\operatorname{Pic}^0(C_y).
\]
Under the trivialization above, $G_{v,y}$ is identified with the fixed
subtorus $S_v^1\subset T^{2g}.$

Moreover, the relative Picard variety $\operatorname{Pic}^{g-1}(\pi)$
trivializes over $U$, and the theta characteristic $L\in \operatorname{Pic}^{g-1}(C)$ extends, after shrinking $U$ if necessary, to a spin section
\[
\mathcal L:U\longrightarrow \operatorname{Pic}^{g-1}(\pi)|_U,
\qquad
\mathcal L(0)=L.
\]
Thus, for every $y\in U$, $\mathcal L(y)\in \operatorname{Pic}^{g-1}(C_y)$
denotes the theta characteristic on $C_y$ determined by this section.

Finally, we consider the $\mathcal L$--even families of tori and define
\begin{equation}\label{cv}
K_v
:=
\left\{
y\in U
\ \middle|\
G_{v,y}\cap\Theta_{\mathcal L(y)}\neq\emptyset
\ \text{ and }\
G_{v,y}\text{ is }\mathcal L(y)\text{--even}
\right\}
\end{equation}
for any
\[
v\in
\operatorname{Hom}\bigl(\mathbb Q,H_1(T^{2g},\mathbb Q)\bigr)
\setminus\{0\}
=
\mathbb Q^{2g}\setminus\{0\}.
\]
Since the natural projection $J(\pi)|_U\longrightarrow U$
is proper, each $K_v$ is closed in $U$.
We argue by contradiction and assume that
\begin{equation}\label{ea}
\bigcup_{\left\{
v\ \middle|\ G_v(\pi)\text{ is }\mathcal L\text{--even}
\right\}}
K_v
=
U,
\end{equation}
where $G_v(\pi)$ is said to be $\mathcal L$--even if
$G_{v,y}$ is $\mathcal L(y)$--even for every $y\in U$.
It follows from \eqref{ea} that $U$ is a countable union of the closed
sets $K_v$. By the Baire Category Theorem, there exist
\[
\overline v\in
\operatorname{Hom}\bigl(\mathbb Q,H_1(T^{2g},\mathbb Q)\bigr)
\setminus\{0\}
=
\mathbb Q^{2g}\setminus\{0\}
\]
and a non-empty open subset $W\subset U$ such that $W\subset K_{\overline v}.$
From now on, in order to simplify the notation, for every $y\in W$ we set 
\[
    G_y:=G_{\overline v,y}\subset J(C_y).
\] 
Thus $G_y$ is the fiber over $y$ of the flat family $G_{\overline v}(\pi)$, and, under the flat trivialization, it is identified with the fixed real one-dimensional subtorus $S^1_{\overline v}\subset T^{2g}$.
By the definition of $K_{\overline v}$, we have 
\[
    G_y\cap\Theta_{\mathcal L(y)}\neq\emptyset \qquad\text{for every }y\in W, 
\] and $G_y$ is $\mathcal L(y)$--even. Consequently, if $\zeta_{G_y}\in G_y$ denotes the unique non-trivial point of order two, then $\mathcal L(y)\otimes\zeta_{G_y}$ is an even theta characteristic on $C_y$. 
Hence, for a point $y\in W$, $h^0\bigl(C_y,\mathcal L(y)\otimes\zeta_{G_y}\bigr)=0.$ We now define the incidence locus
\[
    X:=\bigcup_{y\in W}\left(\bigl(G_y\cap\Theta_{\mathcal L(y)}\bigr)\times\{y\}\right)
    \subset
    S^1_{\overline v}\times W.
\]
Equivalently, $X=\left\{(x,y)\in S^1_{\overline v}\times W
\ \middle|\ x_y\in\Theta_{\mathcal L(y)} \right\},$
where $x_y\in G_y\subset J(C_y)$ denotes the point corresponding to $x\in S^1_{\overline v}$.
Let
\[
    \varpi_1:S^1_{\overline v}\times W\longrightarrow S^1_{\overline v}, \qquad \varpi_2:S^1_{\overline v}\times W\longrightarrow W 
\] 
be the natural projections, and denote by the same symbols their restrictions to $X$.

The incidence locus $X$ is real analytic. To see this, consider the map
\[
\Phi\ :\ S^1_{\overline v}\times W\ \longrightarrow\
\operatorname{Pic}^{g-1}(\pi)|_W,
\qquad \Phi(x,\,y)\ =\ \mathcal L(y)\otimes x.
\]
Then
\[
    X\ =\
    \Phi^{-1}\bigl(\mathcal W_{g-1}(\pi)\bigr),
\]
where $\mathcal W_{g-1}(\pi)$ denotes the relative theta divisor \eqref{ew}.
Since $\Phi$ is real analytic with respect to the real
trivialization of the Jacobian fibration, it follows that $X$ is a real-analytic set.

Now for $\varpi_1(X)$, we have the following two possibilities: Either it is totally disconnected set,
or $\varpi_1(X)$ contains an interval $I \,\subset\, S_{\overline v}^1$.
Hence in both cases, by the fiber-dimension theorem for real-analytic
sets (see \cite{semianalytic} for a more general treatment),
\[
    \dim_{\mathbb R}X\ \le\
    \dim_{\mathbb R}\varpi_1(X)
    +
    \max_{x\in\varpi_1(X)}
    \dim_{\mathbb R}\varpi_1^{-1}(x).
\]
Since the map $\varpi_2\,:\,X\,\longrightarrow\, W$ is surjective, we have $\dim_{\mathbb R}X
\,\ge\, \dim_{\mathbb R}W,$ while $$\dim_{\mathbb R}\varpi_1(X)\ \le\ 1,$$
because $\varpi_1(X)\,\subset\, S^1_{\overline v}$. Therefore, there exists
$x\,\in\, S^1_{\overline v}$ such that
\[
    \dim_{\mathbb R}\varpi_1^{-1}(x)\,\ \ge\,\ \dim_{\mathbb R}W-1.
\]
Equivalently, setting $Z\,:=\,\varpi_2\bigl(\varpi_1^{-1}(x)\bigr)\,\subset\, W,$
we obtain the following:
\[
    \dim_{\mathbb R}Z\,\ \ge\,\ \dim_{\mathbb R}W-1.
\]
By Remark \ref{remark: complx foliation}, $Z\,\subset\, W$ is a complex-analytic subset, $Z$ has full
dimension in $W$ and --- after shrinking $W$ if necessary --- we have $Z \,=\, \varpi_2(\varpi ^{-1}_{1}(x))\, =\, W$.

%%%%%%%%%%%%%%%%%%%%%%%%%%%%%%%
Then there exists a section
\[
    \gamma:W\longrightarrow X\subset S^1_{\overline v}\times W
\]
such that $\varpi_1\bigl(\gamma(y)\bigr)=x$ for every $y\in W.$
Via the flat trivialization, we also denote by $\gamma(y)\in G_y\subset
\operatorname{Pic}^0(C_y)$ the corresponding point in the Jacobian fiber.
By the definition of $X$, we have $\gamma(y)\in\Theta_{\mathcal L(y)}$
for every $y\in W$. Hence there exists an effective divisor $D_y$ of degree $g-1$ such that
\[
\mathcal L(y)\otimes\gamma(y)
\simeq
\mathcal O_{C_y}(D_y).
\]

Therefore, we can define a section $\psi:W\longrightarrow\operatorname{Pic}^{g-1}(\pi)|_W$
by
\[
    \psi(y):=\mathcal L(y)\otimes\gamma(y)\simeq\mathcal O_{C_y}(D_y).
\]
Since $\gamma$ is constant with respect to the flat trivialization, it is
a locally constant section of the Jacobian fibration and, by Remark \ref{remark: complx foliation} it is holomorphic. Hence $\psi$ is a holomorphic section of
$\operatorname{Pic}^{g-1}(\pi)|_W$.

Moreover, since $\mathcal L$ is a spin section, the normal function $\nu$ associated with $\psi$, \eqref{associated normal fun}, is
\[
\nu = 2\psi-\kappa
=
2(\mathcal L+\gamma)-\kappa
=
2\gamma.
\]
Since $\gamma$ is locally constant, the normal function $2\gamma$ is locally
constant as well. Since $D_y$ is effective, we have $h^0(C_y,\psi(y))>0$ for every $y\in W$. Thus, by Theorem~\ref{coro-spin}, $\psi$ is an odd spin section.
On the other hand, $\gamma(y)=\psi(y)\otimes\mathcal L(y)^{-1}$ is a point of order two in $G_y$;
\[
    \gamma(y)^{\otimes 2}
    \simeq
    \psi(y)^{\otimes 2}\otimes\mathcal L(y)^{-2}
    \simeq
    \omega_{C_y}\otimes\omega_{C_y}^{-1}
    \simeq
    \mathcal O_{C_y}.
\]
Moreover, $\gamma(y)$ is non-trivial, since otherwise $\psi(y)\simeq\mathcal L(y),$ which is impossible since $\psi(y)$ is odd and $\mathcal L(y)$ is
even. Therefore $\gamma(y)$ is the unique non-trivial point of order two
of $G_y$, namely
\[
    \gamma(y)=\zeta_{G_y}
    \qquad\text{for every }y\in W.
\]

It follows that $\psi(y)
=
\mathcal L(y)\otimes\zeta_{G_y}$
is an odd theta characteristic. Hence $G_y$ is
$\mathcal L(y)$--odd for every $y\in W$, contradicting the fact that the torus $G_y$ is
$\mathcal L(y)$--even.
\end{proof}

\section{Proof of Theorem \ref{coro-spin}}

In this section, we prove Theorem~\ref{coro-spin}. Although a more general version appears in 
\cite{FassinaPirola}, for completeness we present a self-contained proof for our specific setting. We first 
recall the main ingredients needed for the argument.

Let $\pi\,:\,\mathcal C\,\longrightarrow\, Y$
be as in Section~\ref{sec:2}, and fix a general point $y\,\in\, Y$. Set $C\,:=\,C_y.$
For a divisor $D$ on $C$, denote by
\[
    \operatorname{Supp}(D)
\]
the set of points appearing in $D$ with non-zero coefficient.
We denote by
\[
    \kappa_y\ :\ T_{Y,y}\ \longrightarrow\ H^1(C,\,T_C)
\]
the Kodaira--Spencer map. Since the modular map
\[
    m\ :\ Y\ \longrightarrow\ \mathcal M_g
\]
is dominant, after restricting to a non-empty open subset of $Y$, we
may assume that $\kappa_y$ is surjective.

\subsection{The Griffiths infinitesimal invariant}
Following the notation of Section \ref{sec:2}, set $\mathcal H^{0,1}\,:=\,\mathcal H/\mathcal H^{1,0}.$
Recall that the Gauss--Manin connection
\[
    \nabla\ :\ \mathcal H\ \longrightarrow\
    \mathcal H\otimes\Omega_Y^1
\]
induces, by the Griffiths transversality, a bundle map
\[
    \overline{\nabla}:
    \mathcal H^{1,0}\longrightarrow
    \mathcal H^{0,1}\otimes\Omega_Y^1.
\]
At a point $y\,\in\, Y$, this is given by
\[
    \overline{\nabla}_{\xi}(\omega)\,  =\,
    \kappa_y(\xi)\cdot\omega
\, \in\, H^1(C,\,\mathcal O_C),
    \qquad
    \xi\,\in\, T_{Y,y},\quad
    \omega\,\in\, H^0(C,\,\omega_C),
\]
where ``$\cdot$'' denotes the cup product.

Let
\[
    \nu\ :\ Y\ \longrightarrow\ J(\pi)
\]
be a normal function, and let $\widetilde\nu$ be a local lifting of
$\nu$ to $\mathcal H$. Projecting
$\nabla\widetilde\nu$ to
$\mathcal H^{0,1}\otimes\Omega_Y^1$, and taking its class modulo
$\operatorname{Im}(\overline{\nabla})$, we obtain a well defined class called Griffiths
infinitesimal invariant
\[
    \delta\nu\ \in\
    \frac{\mathcal H^{0,1}\otimes\Omega_Y^1}
         {\operatorname{Im}(\overline{\nabla})}.
\]
For details see, for instance,
\cite[Chapter~2, Section~7]{Voisin}.

For our purpose, we recall the dual description of this invariant.
At $y\,\in\, Y$, consider the map
\[
    \overline{\nabla}^{\,t}_y\ :\
    H^0(C,\,\omega_C)\otimes T_{Y,y}\  \longrightarrow\ H^1(C,\,\mathcal O_C),
\]
given by
\[
    \overline{\nabla}^{\,t}_y
    \left(
        \sum_i \omega_i\otimes\xi_i
    \right)\ =\
    \sum_i \kappa_y(\xi_i)\cdot\omega_i.
\]
By duality, $\delta\nu(y)$ can be regarded as a linear functional on
$\ker\bigl(\overline{\nabla}^{\,t}_y\bigr).$

We shall only need the following property (see \cite[Section 3.1]{Hain}:
\begin{equation}\label{eq:rank-zero-delta-zero}
    \operatorname{rk}(\nu)\, =\, 0
    \quad\Longrightarrow\quad
    \delta\nu\,=\,0.
\end{equation}

\subsection{Schiffer cohomology classes}

Take $ 0\,\neq\,\omega\,\in\, H^0(C,\,\omega_C)$,
and let $Z\, =\, (\omega)_0$
be its zero scheme. Multiplication by $\omega$ gives an exact
sequence
\begin{equation}\label{eq:schiffer-exact-sequence}
    0\, \longrightarrow\, T_C\,
    \xrightarrow{\,\ \cdot\omega\ \, }\,
    \mathcal O_C\,
    \longrightarrow\,
    \mathcal O_Z\,
    \longrightarrow\, 0.
\end{equation}
The corresponding boundary map is
\begin{equation}\label{eb}
    \partial_Z\ :\
    H^0(Z,\,\mathcal O_Z)
\    \longrightarrow\ H^1(C,\,T_C).
\end{equation}
Its image is precisely
\[
    H^1(C,\,T_C)_\omega\ :=\
    \left\{
        \zeta\,\in\, H^1(C,\,T_C)
        \ \big\vert\ \,
        \zeta\cdot\omega\,=\, 0
    \right\}.
\]

The boundary map in \eqref{eb} admits the following Dolbeault description. Given
$h\,\in\, H^0(Z,\,\mathcal O_Z)$, choose a smooth extension of $h$,
supported on a neighbourhood of $Z$. Then
\begin{equation}\label{eq:boundary-dolbeault}
    \partial_Z h\ =\
    \left[
        \frac{\overline{\partial}h}{\omega}
    \right]\   \in\ H^1(C,\,T_C).
\end{equation}

In particular, take $q\,\in\, Z$,
and let $h_q\,\in \,H^0(Z,\,\mathcal O_Z)$ be the section which is equal to
$1$ on the component of $Z$ supported at $q$ and is equal to $0$ on the
other components. We define the Schiffer cohomology class associated
with $q$ and $\omega$ by
\begin{equation}\label{eq:schiffer-class}
    \zeta_q\ :=\
    \partial_Zh_q\ \in\ H^1(C,\,T_C).
\end{equation}
Equivalently, choosing a smooth bump function $\rho_q$ which is equal
to $1$ on a neighbourhood of $q$, and whose support contains no
other point of $Z$, we may write
\[
    \zeta_q\   =\
    \left[
        \frac{\overline{\partial}\rho_q}{\omega}
    \right].
\]
By construction, $\zeta_q\cdot\omega\,=\,0.$

\subsection{Griffiths' formula}

We recall a form of Griffiths' formula that will be used
below (see
\cite[Eq.~(6.18)]{Gri83} and
\cite[Theorem~7.14]{Voisin}). Let $\nu\,:\,Y\,\longrightarrow\, J(\pi)$
be a normal function, and suppose that
\[
    \nu(y)\ =\ \mathcal O_C(A),
    \qquad
    A\ =\ \sum_{i=1}^k a_i p_i,
    \qquad
    \deg A\ =\ 0.
\]
Let $0\,\neq\,\omega\,\in\, H^0(C,\,\omega_C),$ $Z\,=\,(\omega)_0,$
and assume that $\operatorname{Supp}(A)\ \subseteq\ Z.$
Take $\xi \,\in \,T_yY$. If $h \,\in\, H^0(C,\,\mathcal{O}_Z)$ is such that $\partial_Z h
\,=\, \kappa_y(\xi) \,\in\, H^1(C_0, \,T_{C_0})_\omega$,
then
\begin{equation}\label{eq:griffiths-formula}
    \delta\nu(y)(\xi\otimes\omega)\ =\
    \sum_{i=1}^k a_i h(p_i).
\end{equation}
This is the formula of Griffiths for decomposable tensors in the case
where the divisor representing the normal function is supported on the
zero divisor of $\omega$.

\subsection{Proof of Theorem \ref{coro-spin}}

Let $y\,\in\, Y$ be a general point, and set $C\,:=\,C_y$. Since $\dim H^0(C,\,\psi(y))\ >\ 0,$
and $\psi(y)$ has degree $g-1$, we may choose an effective divisor
$D$ on $C$ of degree $g-1$ such that $\psi(y)\,=\,\mathcal O_C(D).$
Write
\[
    D\,=\,\sum_{i=1}^{d}a_i p_i,
    \qquad
    a_i\,>\,0,
    \qquad
    \operatorname{Supp}(D)\,=\,\{p_1,\,\ldots,\,p_d\}.
\]

We will first show that
\begin{equation}\label{es}
d\ =\ g-1.
\end{equation}

To prove \eqref{es} by contradiction, assume that $d\,<\,g-1$. By the Riemann--Roch theorem,
\[
    \dim H^0\left(C,\,\omega_C\left(-\sum_{i=1}^{d}p_i\right)
    \right) \ \geq\ g-d\ \geq\ 2.
\]
Fix a general point $p\,\in\, C\setminus\operatorname{Supp}(D).$
We can then choose a holomorphic differential
\[
    \omega\ \in\
    H^0\left(
        C,\,\omega_C\left(-\sum_{i=1}^{d}p_i\right)
    \right)
\]
such that $\operatorname{ord}_p(\omega)\,=\,1.$
In particular, $\operatorname{Supp}(D)\,\subseteq\, Z.$

Let $\zeta_p\,\in\, H^1(C,\,T_C)$
be the Schiffer cohomology class associated with $p$ and $\omega$.
Since $y$ is general, and the modular map is dominant, the
Kodaira--Spencer map is surjective, and hence there exists $\xi_p\,\in\, T_{Y,y}$
such that
\[
    \kappa_y(\xi_p)\ =\ \zeta_p.
\]
By construction, $\kappa_y(\xi_p)\cdot\omega\,=\,0.$
Set
\[
    K_\omega\ :=\ \operatorname{div}(\omega)
\]
to be the divisor associated to $\omega$.
Since $\nu\,=\,2\psi-\kappa$,
we may represent the value of the normal function at $y$ by
\[
    \nu(y)\ =\
    \mathcal O_C(2D-K_\omega).
\]
The point $p$ does not belong to $\operatorname{Supp}(D)$, whereas
$\operatorname{ord}_p(\omega)\,=\,1$. Therefore, by applying Griffiths'
formula \eqref{eq:griffiths-formula} to $\xi_p\otimes\omega$, we obtain the following:
\[
    \delta\nu(y)(\xi_p\otimes\omega)\ =\ -1.
\]
On the other hand, $\operatorname{rk}(\nu)\,=\,0$, and hence
$\delta\nu(y)\,=\,0$ by \eqref{eq:rank-zero-delta-zero}, so we reached a contradiction.
Therefore, \eqref{es} holds.

Since $D$ is effective, has degree $g-1$, and its support consists
of $g-1$ distinct points, it follows that
\[
    D\ =\ p_1+\cdots+p_{g-1}.
\]

We now show that $2D$ is a canonical divisor. By the Riemann--Roch
theorem,
\[
    \dim H^0(C,\,\omega_C(-D))\ =\ \dim H^0(C,\,\mathcal O_C(D))\ >\ 0.
\]
Hence there exists a non-zero holomorphic differential $\omega\,\in\, H^0(C,\,\omega_C(-D)).$
Write its zero divisor as
\[
    K_\omega\ :=\ \operatorname{div}(\omega)\ =\
    \sum_{i=1}^{g-1}b_i p_i
    +
    \sum_{j=1}^{\ell}c_j q_j,
\]
where $q_j\,\notin\,\{p_1,\,\ldots,\,p_{g-1}\},$ and $b_i,\,c_j\,>\,0.$
Again, we have
\[
    \nu(y)\ =\ \mathcal O_C(2D-K_\omega).
\]

We will prove that that $\ell\,=\,0$.

To prove this by contradiction, suppose that $\ell\,>\,0$. Fix one of the
points $q_j$. Let $\zeta_{q_j}\,\in\, H^1(C,\,T_C)$
be the Schiffer cohomology class associated with $q_j$ and
$\omega$, and choose $\xi_{q_j}\,\in\, T_{Y,y}$
such that $\kappa_y(\xi_{q_j})\,=\,\zeta_{q_j}.$

Since $\nu$ is locally constant, its infinitesimal invariant
vanishes \eqref{eq:rank-zero-delta-zero}. Consequently, Griffiths' formula \eqref{eq:griffiths-formula} gives the following:
\[
    0\ =\
    \delta\nu(y)(\xi_{q_j}\otimes\omega)\ =\
    -c_j.
\]
Therefore, we conclude that $\ell\,=\,0$.

Since $\ell\,=\,0$,
\[
    K_\omega\ =\
    \sum_{i=1}^{g-1}b_i p_i.
\]
It remains to determine the multiplicities $b_i$. For every
$i\,=\,1,\,\ldots,\,g-1$, let $\zeta_{p_i}\,\in\, H^1(C,\,T_C)$
be the Schiffer cohomology class associated with $p_i$ and
$\omega$. Choose $\xi_{p_i}\,\in\, T_{Y,y}$
such that $\kappa_y(\xi_{p_i})\,=\,\zeta_{p_i}.$

The coefficient of $p_i$ in the divisor $ 2D-K_\omega$
is $2-b_i$. Hence Griffiths' formula \eqref{eq:griffiths-formula}, combined with the vanishing of
the infinitesimal invariant, gives that
\[
    0\ =\
    \delta\nu(y)(\xi_{p_i}\otimes\omega)\ =\
    2-b_i.
\]
Therefore, $b_i \,=\, 2$ for every $i\,=\,1,\,\ldots,\,g-1$.
Consequently,
\[
K_\omega\ =\ 2\sum_{i=1}^{g-1}p_i\ =\ 2D.
\]
It follows that
\[
\psi(y)^{\otimes 2}\ =\ \mathcal O_C(2D)\ \simeq\ \omega_C.
\]
Thus $\psi(y)$ is a theta characteristic for a general $y\,\in\, Y$.
Since the equality $2\psi\,=\,\kappa$
holds on a non-empty open subset of $Y$, it holds identically.
Therefore $\psi$ is a spin section. This completes the proof of
Theorem \ref{coro-spin}.

\paragraph{Acknowledgments}

The authors would like to thank the referees for the careful reading of the manuscript and for the helpful comments and suggestions. L. Fassina and G.P. Pirola are members of GNSAGA (INdAM). I. Biswas is partially supported by a J. C. Bose Fellowship (JBR/2023/000003).

\end{document}